\documentclass[11pt]{article}
\usepackage{amssymb}
\usepackage{latexsym,bm}
\usepackage{graphicx}
\usepackage{mathrsfs,amscd,amssymb,amsthm,amsmath,bm,graphicx,psfrag,subfigure,url,xcolor}
\allowdisplaybreaks[3]
\usepackage{amsmath}
\usepackage{mathrsfs}
\usepackage{cite}

\date{}
\newtheorem{lemma}{Lemma}[section]
\newtheorem{theorem}[lemma]{Theorem}

\newtheorem{proposition}[lemma]{Proposition}

\newtheorem{notation}[lemma]{Notation}
\newtheorem{remark}[lemma]{Remark}
\begin{document}
	\title{
		The maximum number of cliques in graphs with given fractional matching number and minimum degree\footnote{E-mail addresses:
			{\tt lichengli0130@126.com}(C. Li),{\tt tyr2290@163.com}(Y. Tang).}}
	\author{\hskip -10mm Chengli Li and Yurui Tang\\
		{\hskip -10mm \small Department of Mathematics, East China Normal University, Shanghai 200241, China}}
	\maketitle
	\begin{abstract}
		Recently, Ma, Qian and Shi determined the maximum size of an $n$-vertex graph with given fractional matching number $s$ and maximum degree at most $d$. Motivated by this result, we determine the maximum number of $\ell$-cliques in a  graph with given fractional matching number and minimum degree, which generalizes Shi and Ma's result about the maximum size of a  graph with given fractional matching number and minimum degree at least one.
		We also determine the maximum number of complete bipartite graphs in a graph with prescribed fractional matching number and minimum degree.
	\end{abstract}
	
	
	\vskip 8mm
	\section{Introduction}
	We consider finite simple graphs and use standard terminology and notations \cite{bondy}. 
	Let $G$ be a graph with vertex set $V(G)$ and edge set
	$E(G)$. We denote the cardinality of the vertex set by $n(G)$ and the cardinality of the edge set by $e(G)$.
	For a vertex $v$ in a graph, we denote by $d(v)$ and $N(v)$ the degree of $v$ and the neighborhood of $v$ in $G$, respectively.
	For vertex disjoint graphs $H$ and $F$, $H+F$ denotes the disjoint union of graphs $H$ and $F$ and $G\vee H$ denotes the join of $G$ and $H$, which is obtained from the disjoint union $G+H$ by adding edges joining every vertex of $G$ to every vertex of $H$. Denote by $\overline{G}$ the complement of a graph $G$. Let $K_{\ell}$ denote the complete graph of order $\ell$ and let $K_{r_{1},r_{2}}$ denote the complete bipartite graph with class sizes $r_{1}$, $r_{2}$.  We denote by $\delta(G)$ the minimum degree and denote by $\Delta(G)$ the maximum degree of a graph $G$.
	Let $\Gamma(v)$ denote the set of edges incident with $v$ in $G$. Let $N(H, G)$ denote the number of copies of $H$ in $G$; e.g., $N(K_2,G)=e(G)$. 
	Given a family of graphs $\mathcal{F},$ let $NM(H, \mathcal{F})=\max\left\{N(H, F) \ |\  F\in \mathcal{F}\right\}$.
	
	A matching is a set of pairwise nonadjacent edges of $G$. A fractional matching of a graph $G$ is a function $f$ assigning each edge with a real
	number in $[0,1]$ so that $\sum_{e\in \Gamma(x)}f(e)\le 1$ for each $x\in V(G)$. The fractional matching number of $G$, denoted by $v^{*}(G)$, is the maximum value of $\sum_{e\in E(G)}f(e)$ over all fractional matchings $f$.
	A matching is a special case of a fractional matching. It is known \cite{Esch} that the fractional matching number is either an integer or a semi-integer, i.e., 2$v^{*}(G)$ is an integer.
	
	Since Tur$\acute{\mathrm {a} }$n proved his well-known theorem in 1941 \cite{Turan}, Tur$\acute{\mathrm {a} }$n-type problems have received a lot of attention\cite{alon,Erdos, Gerbner,Luo,duan,Ning}. In\cite{Erdos}, Erd\H{o}s and Gallai determined the maximum size of an $n$-vertex graph with matching number $k$.
	\begin{theorem}(Erd\H{o}s and Gallai\cite{Erdos})
		\label{eg}
		Let $n,k$ be two positive integers with $n\geq 2k+1$. Let $G$ be an $n$-vertex graph with matching number $k$. Then
		$$e(G)\leq \max\left\{\binom{2k+1}{2},\frac{k(2n-k-1)}{2}\right\}.$$
	\end{theorem}
	
	It is natural to ask the same question by putting
	constraints on the graphs with given matching number. In \cite{chvatal}, Chv$\acute{\mathrm{a}}$tal and Hanson determined
	the maximum size of graphs with given matching number $k$ and maximum degree at most $d$.
	By using the shifting method, Wang \cite{wang} determined the
	maximum number of copies of $K_{\ell}$  in an $n$-vertex graph with given matching number. In \cite{duan}, Duan, Ning, Peng, Wang and Yang determined
	the maximum number of cliques in graphs  with given minimum degree and matching number at most $k$.
	Recently, Liu and Zhang \cite{zll} determined the maximum number of copies of $K_{r_1,\dots,r_s}$ in graphs with given matching number and minimum degree at least $k$.
	
	Along these results, Ma, Qian and Shi \cite{Ma} determined the maximum size of an $n$-vertex graph with fractional matching number $s$ and maximum degree at most $d$. As a corollary, they obtained the maximum size of graphs with a given fractional matching number.
	
	\begin{theorem}(Ma, Qian and Shi \cite{Ma})
		Let $n$, $2s$ and $d$ be positive integers with $n>2s$.
		Denote by $f(n,s,d)=\max \{e(G):n(G)=n,v^{*}(G)=s,\Delta(G)\leq d\}$.
		If $2s$ is even, then
		$$f(n,s,d)= \begin{cases}
			\max\left \{ \begin{pmatrix}
				2s \\2\end{pmatrix},\left \lfloor \frac{s(n+d-s)}{2}  \right \rfloor  \right \}  &  \text{ if }  d\geq 2s-1, n\leq d+s;\\
			ds &  otherwise.
		\end{cases}$$
		If $2s$ is odd, then
		$$f(n,s,d)= \begin{cases}
			\max\left \{ \begin{pmatrix}
				2s \\2\end{pmatrix},\left \lfloor \frac{(s-\frac{3}{2} )(n+d-s+\frac{3}{2} )}{2}  \right \rfloor +3 \right \}  &  \text{ if } d\geq 2s-1, n\leq d+s-\frac{3}{2};\\
			\max \left \{\begin{pmatrix}  2s \\2\end{pmatrix}, d(s-\frac{3}{2})+3\right \}  & \text{ if } d\geq 2s-1, n\geq d+s-\frac{3}{2};\\
			\left \lfloor ds  \right \rfloor & \text{ if } d\leq 2s-1.
		\end{cases}$$
	\end{theorem}
	It is natural to consider graphs with given fractional matching number and minimum degree. Recently, Shi and Ma \cite{Shi} determined the maximum size of an $n$-vertex  graph with a given fractional matching number.
	
	\begin{notation}
		Let $n$, $2s$ and $\ell$ be positive integers with $n \geq 2s+1\geq5$ and $\ell\geq 2$. Given a positive integer $\delta$, for any integer $t$ with $\delta \le t\le s$, denote by $G(n,s,t)$ the graph obtained from $K_{t}\vee\left(K_{2s-2t}+\overline{K_{n+t-2s}}\right)$ by deleting $t-\delta$ edges that are incident to one common vertex $u$ in $\overline{K_{n+t-2s}}$.
		Denote by $g_{\ell}(n,s,t)$ the number of copies of $K_{\ell}$ in $G(n,s,t)$.
	\end{notation}
	Let $A$ be the set of vertices whose degree is at most $t$, $C$ be the set of vertices whose degree is at least $n-2$ in $G(n,s,t)$ and let $B=V(G(n,s,t))\setminus (A\cup C)$. Note that the number of copies of $K_{\ell}$ in $B\cup C$ is $\begin{pmatrix}2s-t\\\ell\end{pmatrix}$, the number of copies of $K_{\ell}$ that contains a vertex in $A\setminus\left\{u\right\}$ and does not contain the vertex $u$ is $\begin{pmatrix}t\\\ell-1\end{pmatrix}(n-2s+t-1)$ and the number of copies of $K_{\ell}$ that contains the vertex $u$ is $\begin{pmatrix}\delta\\\ell-1\end{pmatrix}$. Therefore,
	\begin{equation}\label{1}
		\begin{aligned}
			g_{\ell}(n,s,t)=\begin{pmatrix}2s-t\\\ell\end{pmatrix}+\begin{pmatrix}t\\\ell-1\end{pmatrix}(n+t-2s-1)+\begin{pmatrix}\delta \\\ell-1\end{pmatrix}.
		\end{aligned}
	\end{equation}

	\begin{theorem}(Shi and Ma \cite{Shi})
		Let $n$, $2s$ be positive integers with $n\geq 2s+1\geq 5$. Let $G$ be a graph of order $n$ with fractional matching number $s$ and minimum degree at least one. \\
		If $2s$ is even, then
		$$e(G)\leq \max\left\{	\begin{pmatrix}2s-2\\2\end{pmatrix}+n-1, \begin{pmatrix}s\\2\end{pmatrix}+s(n-s)\right\}.$$
		If $2s$ is odd, then
		$$e(G)\leq \max\left\{\begin{pmatrix}2s-2\\2\end{pmatrix}+n-1,\begin{pmatrix}s-\frac{3}{2}\\2\end{pmatrix}+3+(s-\frac{3}{2})(n-s+\frac{3}{2})\right\}.$$
	\end{theorem}
	\begin{remark}
		The above equality holds if $\delta=1$ and $G=G(n,s,1)$ or
		$\delta=s$ and $G=G(n,s,s)$ when $2s$ is even, if $\delta=1$ and $G=G(n,s,1)$ or
		$\delta=s-\frac{3}{2}$ and $G=G(n,s,s-\frac{3}{2})$ when $2s$ is odd.
	\end{remark}
	Motivated by the above results, we determine the maximum number of copies of $K_{\ell}$ in $n$-vertex graphs with prescribed fractional matching number $s$ and minimum degree $\delta$.
	
	\begin{theorem}\label{kl}
		Let $n$, $2s$, $\delta$ and $\ell$ be positive integers with $n \geq 2s+1\geq5$ and $\ell \geq 2$. Let $G$ be a graph of order n with  fractional matching number $s$ and minimum degree $\delta$.\\
		If $2s$ is even, then
		$$N(K_{\ell},G)\leq \max\left\{g_{\ell}(n,s,\delta),g_{\ell}(n,s,s)\right\}.$$
		If 2s is odd, then
		$$N(K_{\ell},G)\leq \max \left\{g_{\ell}(n,s,\delta),g_{\ell}\left(n,s,s-\frac{3}{2}\right) \right\}.$$
	\end{theorem}
	
	Theorem~\ref{kl} is sharp as shown by the following remark.
	
	\begin{remark}
		Equality in Theorem 1.6 holds if the following condition holds:\\
		$(1)$ If $G=G(n,s,\delta)$, $N(K_{\ell},G)=g_{\ell}(n,s,\delta)$;\\
		$(2)$ If $2s$ is even and $G=G(n,s,s)$, $N(K_{\ell},G)=g_{\ell}(n,s,s)$;\\
		$(3)$ If $2s$ is odd and $G=G\left(n,s,s-\frac{3}{2}\right)$, $N(K_{\ell},G)=g_{\ell}\left(n,s,s-\frac{3}{2}\right)$.
	\end{remark}
	
	Moreover, we find a lot of work on the maximum number of copies of $K_{r_{1},r_{2}};$ see \cite{wang,zlw,zll2,12}. In\cite{wang}, Wang determined the maximum number of copies of $K_{r_{1},r_{2}}$ in bipartite graphs with a given matching number. In\cite{zll2}, Zhang determined the maximum number of copies of $K_{r_{1},r_{2}}$ in an $n$-vertex graph with given maximum size of linear forest and the minimum degree. Motivated by their work, we determine the maximum number of copies of $K_{r_{1},r_{2}}$ with prescribed fractional matching number and minimum degree.
	
	\begin{notation}
		Let $n$, $2s$, $r_1$ and $r_2$ be positive integers with $n \geq 2s+1\geq5$. Denote by $g_{r_1,r_2}(n,s,t)$ the number of copies of $K_{r_1,r_2}$ in $G(n,s,t)$, where $G(n,s,t)$ is defined in Notation 1.3.
	\end{notation}
	Suppose that $r=r_1+r_2$.
	Let $c=1$ if $r_{1}\neq r_{2}$, and $c=2$  if $r_{1}=r_{2}$.
	Let $A$ be the set of vertices whose degree is at most $t$, $C$ be the set of vertices whose degree is at least $n-2$ in $G(n,s,t)$ and let $B=V(G(n,s,t))\setminus (A\cup C)$.
	Note that the number of copies of $K_{r_{1},r_{2}}$ in $B\cup C$ is $\frac{1}{c}\begin{pmatrix}2s-t\\r\end{pmatrix}\begin{pmatrix}r\\r_{1}\end{pmatrix}$, the number of copies of $K_{r_{1},r_{2}}$ containing the vertex $u$ is $\frac{1}{c}\sum\limits_{j=1}^{2}\begin{pmatrix}\delta\\r_{j}\end{pmatrix}\begin{pmatrix}n-r_j-1\\r-r_{j}-1\end{pmatrix}$ and 
	for one partite set of $K_{r_{1},r_{2}}$ is in $C$, the other partite set contains a vertex in $A\setminus \left\{u\right\}$ and does not contain the vertex $u$, the number of copies of $K_{r_{1},r_{2}}$ is $\frac{1}{c}\sum\limits_{j=1}^{2}\begin{pmatrix}t\\r_{j}\end{pmatrix}\left[\begin{pmatrix}n-r_{j}-1\\r-r_{j}\end{pmatrix}-\begin{pmatrix}2s-t-r_{j}\\r-r_{j}\end{pmatrix}\right]$. Therefore,
	\begin{equation}\label{2}
		\begin{aligned}
			g_{r_1,r_2}(n,s,t)=&\frac{1}{c}\sum\limits_{j=1}^{2}\begin{pmatrix}t\\r_{j}\end{pmatrix}\left[\begin{pmatrix}n-r_{j}-1\\r-r_{j}\end{pmatrix}-\begin{pmatrix}2s-t-r_{j}\\r-r_{j}\end{pmatrix}\right]+ \frac{1}{c}\sum\limits_{j=1}^{2}\begin{pmatrix}\delta\\r_{j}\end{pmatrix}\begin{pmatrix}n-r_j-1\\r-r_{j}-1\end{pmatrix}\\
			&+\frac{1}{c}\begin{pmatrix}2s-t\\r\end{pmatrix}\begin{pmatrix}r\\r_{1}\end{pmatrix}.
		\end{aligned}
	\end{equation}
	
	\begin{theorem}\label{kr}
		Let $n$, $2s$, $\delta$, $r_1$ and $r_2$ be positive integers with $n \geq 2s+1\geq5$. Let $G$ be a graph of order $n$ with fractional matching number $s$ and minimum degree $\delta$.\\
		If $2s$ is even, then
		$$N(K_{r_{1},r_{2}},G)\leq \max\left\{ g_{r_1,r_2}(n,s,\delta),g_{r_1,r_2}(n,s,s)\right\}.$$
		If $2s$ is odd, then
		$$N(K_{r_{1},r_{2}},G)\leq \max\left\{ g_{r_1,r_2}(n,s,\delta),g_{r_1,r_2}\left(n,s,s-\frac{3}{2}\right)\right\}.$$
	\end{theorem}
	
	Theorem~\ref{kr} is sharp as shown by the following remark.
	
	\begin{remark}
		Equality in Theorem 1.9 holds if the following condition holds:\\
		$(1)$ If $G=G(n,s,\delta)$, $N(K_{r_{1},r_{2}},G)=g_{r_1,r_2}(n,s,\delta)$;\\
		$(2)$ If $2s$ is even and $G=G(n,s,s)$, $N(K_{r_{1},r_{2}},G)=g_{r_1,r_2}(n,s,s)$;\\
		$(3)$ If $2s$ is odd and $G=G\left(n,s,s-\frac{3}{2}\right)$, $N(K_{r_{1},r_{2}},G)=g_{r_1,r_2}\left(n,s,s-\frac{3}{2}\right)$.
	\end{remark}
	
	\section{Proof of the results}
	To prove Theorem \ref{kl} and Theorem \ref{kr}, we first need a well-known result, called fractional Tutte-Berge formula.
	\begin{theorem}\cite{Esch}
		Let $G$ be a graph of order $n$. Then
		$$v^{*}(G)=\frac{1}{2} \left(n-\max_{T\subseteq V(G)}\left\{i(G-T)-\left |  T\right | \right \}\right)$$
		where $i(G-T)$ is the number of isolated vertices in $G-T$.
	\end{theorem}
	The following Pascal's Rule is useful throughout our proof. For any positive integers $n$, $m$ with $n\ge m$, we have $\begin{pmatrix}n+1\\m\end{pmatrix}=\begin{pmatrix}n\\m\end{pmatrix}+\begin{pmatrix}n\\m-1\end{pmatrix}$.
	
	Let $n$, $2s$ be positive integers with $n \geq 2s+1\geq5$.
	Given a positive integer $\delta$, for any integer $t$ with $\delta \le t\le  s  $,
	we denote
	\[
	\mathcal{F}_1(t)=\{K_t\vee (K_{2s-2t}+ \overline{K_{n+t-2s}})-E_1\ |\
	E_1\subseteq \Gamma(v),|E_1|=2s-t-1-\delta,~\text{where} ~v\in V(K_{2s-2t})\}
	\]
	and
	\[
	\mathcal{F}_2(t)=\{K_t\vee (K_{2s-2t}+ \overline{K_{n+t-2s}})-E_2\ |\ E_2\subseteq \Gamma(v),\ |E_2|=n-1-\delta,~\text{where}~v\in V(K_t)\}.
	\]
	
	To prove Theorem~\ref{kl}, we need the following lemma and proposition.
	\begin{lemma}
		Let $n$, $2s$ and $\ell$ be positive integers with $n \geq 2s+1\geq5$. Given a positive integer $\delta$, we have $$N(K_{\ell},G(n,s,t))\ge NM(K_{\ell},\mathcal{F}_1(t))~~\text{for}~~\delta \le t \le s-1 $$
		and $$N(K_{\ell},G(n,s,t))\ge NM(K_{\ell},\mathcal{F}_2(t))~~\text{for}~~\delta \le t \le  s.$$
	\end{lemma}
	\begin{proof}[\rm{\textbf{Proof}}]
		Let $G_{i}(n,s,t)$ be the graph attaining the maximum number of copies of $K_\ell$ in $\mathcal{F}_i(t)$ for $i=1,2$. 
		
		First, we prove that $N(K_{\ell},G(n,s,t))\ge NM(K_{\ell},\mathcal{F}_1(t))$ for $\delta\le t\le s-1$. 
		Let $A_1$ be the set of vertices whose degree is at most $t$, $C_1$ be the set of vertices whose degree is at least $n-2$ in $G_1(n,s,t)$ and let $B_1=V(G_1(n,s,t))\setminus (A_1\cup C_1)$. 
		
		Note that the number of copies of $K_{\ell}$ in $B_1\cup C_1$ is $\begin{pmatrix}2s-t-1\\\ell\end{pmatrix}$, the number of copies of $K_{\ell}$ that contains a vertex in $A_1\setminus\left\{v\right\}$  is $\begin{pmatrix}t\\\ell-1\end{pmatrix}(n-2s+t)$ and the number of copies of $K_{\ell}$ that contains the vertex $v$ is $\begin{pmatrix}\delta\\\ell-1\end{pmatrix}$. Thus,
		
		$$N(K_{\ell},G_{1}(n,s,t))=\begin{pmatrix}2s-t-1\\\ell\end{pmatrix}+\begin{pmatrix}\delta\\\ell-1\end{pmatrix}+\begin{pmatrix}t\\\ell-1\end{pmatrix}(n-2s+t).$$
		Since $t\leq s-1$, combining with Eq.~\eqref{1}, we have
		$$N(K_{\ell},G(n,s,t))-N(K_{\ell},G_{1}(n,s,t))\geq\begin{pmatrix}2s-t-1\\\ell-1\end{pmatrix}-\begin{pmatrix}t\\\ell-1\end{pmatrix}\ge 0. $$
		
		Next we prove that $N(K_{\ell},G(n,s,t))\geq NM(K_{\ell},\mathcal{F}_2(t))$ for $\delta\le t\le s$.
		Let $A_2$ be the set of vertices whose degree is at most $t$, $C_2$ be the set of vertices whose degree is at least $n-2$ in $G_{2}(n,s,t)$ and let $B_2=V(G_{2}(n,s,t))\setminus (A_2\cup C_2)$.
		
		Recall that $v$ is the vertex with degree $\delta$ in $G_{2}(n,s,t)$. 
		Then the number of copies of $K_{\ell}$ that contains the vertex $v$ is at most $\begin{pmatrix}\delta\\\ell-1\end{pmatrix}$, the number of copies of $K_{\ell}$ that contains a vertex in $A_2\setminus \{v\}$ and does not contain the vertex $v$ is $\begin{pmatrix}t-1\\\ell-1\end{pmatrix}(n-2s+t)$,
		and the number of copies of $K_{\ell}$ that does not contain the vertex $v$ in $B_2\cup C_2$ is $\begin{pmatrix}2s-t-1\\\ell\end{pmatrix}$. Thus,
		$$N(K_{\ell},G_{2}(n,s,t))\le \begin{pmatrix}2s-t-1\\\ell\end{pmatrix}+\begin{pmatrix}\delta\\\ell-1\end{pmatrix}+\begin{pmatrix}t-1\\\ell-1\end{pmatrix}(n-2s+t).$$
		Therefore, combining with Eq.~\eqref{1}, we have
		\begin{align*}
			&\quad N(K_{\ell},G(n,s,t))-N(K_{\ell},G_{2}(n,s,t))\\
			&\ge \begin{pmatrix}2s-t-1\\\ell-1\end{pmatrix}-\begin{pmatrix}t-1\\\ell-1\end{pmatrix}
			+\left(\begin{pmatrix}t\\\ell-1\end{pmatrix}-\begin{pmatrix}t-1\\\ell-1\end{pmatrix}\right)(n-2s+t-1)\\
			&\ge 0
		\end{align*}
		where the last inequality follows as
		$s\geq t$ and $n\geq 2s+1$.
	\end{proof}
	
	\begin{lemma}
		Let $\ell$ and $2s$ be positive integers. For positive integer $t$ with $t\leq 2s$, $\begin{pmatrix}2s-t\\\ell\end{pmatrix}$ is a convex function of $t$.
	\end{lemma}
	
	\begin{lemma}
		Let $n$, $2s$ and $\ell$ be positive integers. For positive integer $t$, $f(t)= \begin{pmatrix}t\\\ell-1\end{pmatrix}(n+t-2s-1)$ is a convex function of $t$.
	\end{lemma}
	
	\begin{proof}[\rm{\textbf{Proof}}]
		By direct calculation, we have
		\begin{align*}
			&\quad f(t+1)+f(t-1)-2f(t)\\
			&=\begin{pmatrix}t+1\\\ell-1\end{pmatrix}(n+t-2s)+\begin{pmatrix}t-1\\\ell-1\end{pmatrix}(n+t-2s-2)-2\begin{pmatrix}t\\\ell-1\end{pmatrix}(n+t-2s-1)\\
			&=\left(\begin{pmatrix}t+1\\\ell-1\end{pmatrix}+\begin{pmatrix}t-1\\\ell-1\end{pmatrix}-2\begin{pmatrix}t\\\ell-1\end{pmatrix}\right)(n+t-2s-1)+\begin{pmatrix}t+1\\\ell-1\end{pmatrix}-\begin{pmatrix}t-1\\\ell-1\end{pmatrix}\\
			&=\left(\begin{pmatrix}t\\\ell-2\end{pmatrix}-\begin{pmatrix}t-1\\\ell-2\end{pmatrix}\right)(n+t-2s-1)+\begin{pmatrix}t+1\\\ell-1\end{pmatrix}-\begin{pmatrix}t-1\\\ell-1\end{pmatrix}\\
			&\geq 0.
		\end{align*}
		This implies that $f(t)$ is a convex function of $t$.
	\end{proof}
	By Lemmas~2.3 and 2.4, we have the following proposition.
	\begin{proposition}\label{convex1}
		Let $n$, $2s$, $\delta$ and $\ell$ be positive integers. For positive integer $t$ with $t\leq s  $, $$g_{\ell}(n,s,t)=\begin{pmatrix}2s-t\\\ell\end{pmatrix}+\begin{pmatrix}t\\\ell-1\end{pmatrix}(n+t-2s-1)+\begin{pmatrix}\delta \\\ell-1\end{pmatrix}$$
		is a convex function of $t$.
	\end{proposition}
	\begin{proof}[\rm{\textbf{Proof of Theorem~\ref{kl}}}]
		Let $G$ be a graph attaining the maximum number of copies of $K_\ell$ with fractional matching number $s$ and minimum degree $\delta$.
		
		By fractional Tutte-Berge formula, it is not hard to see that $G$ is a subgraph of $K_t\vee (K_{2s-2t}+ \overline{K_{n+t-2s}})$, with $t\le s$ and $2s-2t\neq 1$.
		Since $\delta(G)=\delta$, it is clear that $\delta\le t$, and  hence $G$ is a subgraph of $G(n,s,t)$, $G_{1}$ or $G_{2}$, where $G_{1}\in \mathcal{F}_1(t)$, $G_{2}\in \mathcal{F}_2(t)$. 
		Note that deleting any edge of a graph does not increase the number of copies of $K_\ell$. So we may assume that $G=G(n,s,t)$, $G\in \mathcal{F}_1(t)$ or
		$G\in \mathcal{F}_2(t)$.
		
		In particular, if $G\in \mathcal{F}_1(s)$, then $\delta=s$ and hence $G=G(n,s,s)$.
		By the maximality of $G$ and Lemma 2.2, we may assume that $G=G(n,s,t)$ for some positive integer $t$ with $\delta\leq t\leq s$.
		
		{\bf Case 1. $2s$ is even.}
		Now $s$ is a positive integer. By Proposition~\ref{convex1}, we have $t=s$ or $t=\delta$.
		If $t=s$, then $G=G(n,s,s)$ and hence $N(K_{\ell},G)=g_{\ell}(n,s,s).$
		If $t=\delta$, then $G=G(n,s,\delta)$ and hence $N(K_{\ell},G)=g_{\ell}(n,s,\delta).$
		
		{\bf Case 2. $2s$ is odd.}
		Since $t\neq s-\frac{1}{2}$, we have $\delta \leq t \leq s-\frac{3}{2}$. By Proposition~\ref{convex1}, we have $t=\delta$ or $t=s-\frac{3}{2}$.
		If $t=s-\frac{3}{2}$, then $G=G(n,s,s-\frac{3}{2})$ and hence $N(K_{\ell},G)=g_{\ell}(n,s,s-\frac{3}{2})$.
		If $t=\delta$, then $G=G(n,s,\delta)$ and hence $N(K_{\ell},G)=g_{\ell}(n,s,\delta)$.
		
		This completes the proof.
	\end{proof}
	Next, to prove Theorem~\ref{kr}, we need the following lemma and proposition.
	
	\begin{lemma}
		Let $n$, $2s$, $r_1$ and $r_2$ be  positive integers with $n \geq 2s+1\geq5$. Given a positive integer $\delta$, $$N(K_{r_1,r_2},G(n,s,t))\geq NM(K_{r_1,r_2},\mathcal{F}_1(t))~~\text{for}~~\delta \le t \le s-1 $$
		and $$N(K_{r_1,r_2},G(n,s,t))\geq NM(K_{r_1,r_2},\mathcal{F}_2(t))~~\text{for}~~\delta \le t \le s.$$
	\end{lemma}
	
	\begin{proof}[\rm{\textbf{Proof}}]
		Let $G_{i}(n,s,t)$ be a graph with the maximum number of copies of $K_{r_1,r_2}$ in $\mathcal{F}_{i}(t)$ for $i=1,2$. 
		Let $r=r_{1}+r_{2}$. Consider the case $r_{1}\neq r_{2}$. 
		
		First we prove that $N(K_{r_1,r_2},G(n,s,t))\geq NM(K_{r_1,r_2},\mathcal{F}_1(t))$ for $\delta \le t \le s-1 $. Let $A_1$ be the set of vertices whose degree is at most $t$, $C_1$ be the set of vertices whose degree is at least $n-2$ in $G_{1}(n,s,t)$ and let $B_1=V(G_{1}(n,s,t))\setminus (A_1\cup C_1)$. We determine the value of $N(K_{r_1,r_2},G_{1}(n,s,t))$.
		Recall that $v$ is the vertex with minimum degree $\delta$ in $G_{1}(n,s,t)$. 
		Then the number of copies of $K_{r_{1},r_{2}}$ that contains the vertex $v$ is at most $\sum\limits_{j=1}^{2}\begin{pmatrix}\delta\\r_{j}\end{pmatrix}\begin{pmatrix}n-r_j-1\\r-r_{j}-1\end{pmatrix}.$ 
		In $B_1\cup C_1$, the number of copies of $K_{r_{1},r_{2}}$ is $\begin{pmatrix}2s-t-1\\r\end{pmatrix}\begin{pmatrix}r\\r_{1}\end{pmatrix}$. 
		For one partite set of $K_{r_{1},r_{2}}$ is in $C_1$, the other partite set contains a vertex in $A_1\setminus \{v\}$ and does not contain the vertex $v$,  the number of copies of $K_{r_{1},r_{2}}$  is $\sum\limits_{j=1}^{2}\begin{pmatrix}t\\r_{j}\end{pmatrix}\left[\begin{pmatrix}n-r_{j}-1\\r-r_{j}\end{pmatrix}-\begin{pmatrix}2s-t-r_{j}-1\\r-r_{j}\end{pmatrix}\right]$. Thus
		\begin{align*}
			N(K_{r_1,r_2},G_{1}(n,s,t))
			&\leq
			\sum\limits_{j=1}^{2}\begin{pmatrix}t\\r_{j}\end{pmatrix}\left[\begin{pmatrix}n-r_{j}-1\\r-r_{j}\end{pmatrix}-\begin{pmatrix}2s-t-r_{j}-1\\r-r_{j}\end{pmatrix}\right]\\
			&\quad +\begin{pmatrix}2s-t-1\\r\end{pmatrix}\begin{pmatrix}r\\r_{1}\end{pmatrix}+\sum\limits_{j=1}^{2}\begin{pmatrix}\delta\\r_{j}\end{pmatrix}\begin{pmatrix}n-r_j-1\\r-r_{j}-1\end{pmatrix}.
		\end{align*}
		Combining with Eq.~\eqref{2}, we have  $$N(K_{r_1,r_2},G(n,s,t))-N(K_{r_1,r_2},G_{1}(n,s,t))
		\geq\begin{pmatrix}2s-t-1\\r-1\end{pmatrix}\begin{pmatrix}r\\r_{1}\end{pmatrix}-\sum\limits_{j=1}^{2}\begin{pmatrix}t\\r_{j}\end{pmatrix}\begin{pmatrix}2s-t-r_{j}-1\\r-r_{j}-1\end{pmatrix}.$$
		Note that
		$$\begin{pmatrix}2s-t-1\\r-1\end{pmatrix}\begin{pmatrix}r\\r_{1}\end{pmatrix}\geq \sum\limits_{j=1}^{2}\begin{pmatrix}t\\r_{j}\end{pmatrix}\begin{pmatrix}2s-t-r_{j}-1\\r-r_{j}-1\end{pmatrix}$$
		$$\Longleftrightarrow \qquad \frac{(2s-t-1)!\cdot r!}{(r-1)!\cdot (2s-t-r)!\cdot(r-r_{1})!\cdot r_{1}!} \geq \sum\limits_{j=1}^{2}\frac{t!\cdot (2s-t-r_{j}-1)!}{r_{j}!\cdot (t-r_{j})!\cdot (r-r_{j}-1)!\cdot (2s-t-r)!}\quad$$
		$$\Longleftrightarrow \qquad \quad\quad(2s-t-1)!\cdot r\geq\frac{t!\cdot(2s-t-r_{1}-1)!\cdot r_{2}}{(t-r_{1})!}+\frac{t!\cdot(2s-t-r_{2}-1)!\cdot r_{1}}{(t-r_{2})!}\quad\quad\quad\qquad$$
		$$\Longleftrightarrow  1\geq\frac{t(t-1)\cdots(t-r_{1}+1)\cdot r_{2}}{(2s-t-1)(2s-t-2)\cdots(2s-t-r_{1})\cdot r}+\frac{t(t-1)\cdots(t-r_{2}+1)\cdot r_{1}}{(2s-t-1)(2s-t-2)\cdots(2s-t-r_{2})\cdot r}.$$
		Since $t\leq s-1$, we have $2s-t-1\ge t$, and hence 
		$$\frac{t(t-1)\cdots(t-r_{j}+1)}{(2s-t-1)(2s-t-2)\cdots(2s-t-r_{j})}\leq 1~\text{with}~j=1,2.$$ Recall that $r_{1}+r_{2}=r$.
		Therefore, $N(K_{r_1,r_2},G(n,s,t))\geq NM(K_{r_1,r_2},\mathcal{F}_1(t))~~\text{for}~~\delta \le t \le s-1$.
		
		Next we prove that $N(K_{r_1,r_2},G(n,s,t))\geq NM(K_{r_1,r_2},\mathcal{F}_2(t))~~\text{for}~~\delta \le t \le s-\frac{1}{2}.
		$ Let $A_2$ be the set of vertices whose degree is at most $t$, $C_2$ be the set of vertices whose degree is at least $n-2$ in $G_{2}(n,s,t)$ and let $B_2=V(G_{2}(n,s,t))\setminus (A_2\cup C_2)$.
		We determine the value of $N(K_{r_1,r_2},G_{2}(n,s,t))$. 
		In $B_2\cup C_2$, the number of copies of $K_{r_{1},r_{2}}$ is $\begin{pmatrix}2s-t-1\\r\end{pmatrix}\begin{pmatrix}r\\r_{1}\end{pmatrix}$. Recall that $v$ is the vertex with minimum degree $\delta$ in $G_2(n,s,t)$. Then the number of copies of $K_{r_{1},r_{2}}$ that contains the vertex $v$ is at most $\sum\limits_{j=1}^{2}\begin{pmatrix}\delta\\r_{j}\end{pmatrix}\begin{pmatrix}n-r_j-1\\r-r_{j}-1\end{pmatrix}$. For one partite set of $K_{r_{1},r_{2}}$ is in $C_2$, the other partite set contains a vertex in $A_2\setminus \{v\}$ and does not contain the vertex $v$, the number of copies of $K_{r_{1},r_{2}}$ is $\sum\limits_{j=1}^{2}\begin{pmatrix}t-1\\r_{j}\end{pmatrix}\left[\begin{pmatrix}n-r_{j}-1\\r-r_{j}\end{pmatrix}-\begin{pmatrix}2s-t-r_{j}-1\\r-r_{j}\end{pmatrix}\right]$.
		Thus,
		\begin{align*}
			N(K_{r_1,r_2},G_{2}(n,s,t))
			&\leq \sum\limits_{j=1}^{2}\begin{pmatrix}t-1\\r_{j}\end{pmatrix}\left[\begin{pmatrix}n-r_{j}-1\\r-r_{j}\end{pmatrix}-\begin{pmatrix}2s-t-r_{j}-1\\r-r_{j}\end{pmatrix}\right]\\
			&\quad+\begin{pmatrix}2s-t-1\\r\end{pmatrix}\begin{pmatrix}r\\r_{1}\end{pmatrix}+\sum\limits_{j=1}^{2}\begin{pmatrix}\delta\\r_{j}\end{pmatrix}\begin{pmatrix}n-r_j-1\\r-r_{j}-1\end{pmatrix}.
		\end{align*}
		Therefore, combining with Eq.~\eqref{2}, we have
		\begin{align*}
			&\quad N(K_{r_1,r_2},G(n,s,t))-N(K_{r_1,r_2},G_2(n,s,t))\\
			&\geq \sum\limits_{j=1}^{2}\begin{pmatrix}t-1\\r_{j}-1\end{pmatrix}\begin{pmatrix}n-r_j-1\\r-r_j\end{pmatrix}-\sum\limits_{j=1}^{2}\begin{pmatrix}t\\r_{j}\end{pmatrix}\begin{pmatrix}2s-t-r_j\\r-r_j\end{pmatrix}	+\sum\limits_{j=1}^{2}\begin{pmatrix}t-1\\r_{j}\end{pmatrix}\begin{pmatrix}2s-t-r_j-1\\r-r_j\end{pmatrix}\\
			&\quad+\begin{pmatrix}2s-t-1\\r-1\end{pmatrix}\begin{pmatrix}r\\r_1\end{pmatrix}\\
			&\geq\sum\limits_{j=1}^{2}\begin{pmatrix}t-1\\r_{j}-1\end{pmatrix}\begin{pmatrix}2s-t-r_j-1\\r-r_j\end{pmatrix}+\sum\limits_{j=1}^{2}\begin{pmatrix}t-1\\r_{j}\end{pmatrix}\begin{pmatrix}2s-t-r_j-1\\r-r_j\end{pmatrix}
			-\sum\limits_{j=1}^{2}\begin{pmatrix}t\\r_{j}\end{pmatrix}\begin{pmatrix}2s-t-r_j\\r-r_j\end{pmatrix}\\
			&\quad+\begin{pmatrix}2s-t-1\\r-1\end{pmatrix}\begin{pmatrix}r\\r_1\end{pmatrix}\\
			&\geq\begin{pmatrix}2s-t-1\\r-1\end{pmatrix}\begin{pmatrix}r\\r_{1}\end{pmatrix}-\sum\limits_{j=1}^{2}\begin{pmatrix}t\\r_{j}\end{pmatrix}\begin{pmatrix}2s-t-r_{j}-1\\r-r_{j}-1\end{pmatrix}\\
			&\geq 0.
		\end{align*}
		The second inequality follows as $n>2s-t$ and the last inequality holds by a similar discussion as above.
		Hence, $N(K_{r_1,r_2},G(n,s,t))\geq N(K_{r_1,r_2},G_2(n,s,t))$ for $\delta\le t\le s-\frac{1}{2}$.
		
		Suppose that $t=s$.  
		In this case, $G_2(n,s,s)-v$ is a subgraph of $G(n,s,s)-u$. Hence, $N(K_{r_1,r_2},G_2(n,s,s)-v)\leq N(K_{r_1,r_2},G(n,s,s)-u)$. Note that the number of copies of $K_{r_1,r_2}$ containing the vertex $v$ in $G_2(n,s,s)$ is at most $\sum\limits_{j=1}^{2}\begin{pmatrix}\delta\\r_{j}\end{pmatrix}\begin{pmatrix}n-r_j-1\\r-r_{j}-1\end{pmatrix}$, which is exactly the number of copies of $K_{r_1,r_2}$ containing the vertex $u$ in $G(n,s,s)$. Thus $N(K_{r_1,r_2},G_2(n,s,s))\leq N(K_{r_1,r_2},G(n,s,s)).$
		
		Therefore, we have $N(K_{r_1,r_2},G(n,s,t))\geq NM(K_{r_1,r_2},\mathcal{F}_2(t))~~\text{for}~~\delta \le t \le s.$
		
		For the case $r_1=r_2$,
		by the same discussion and deleting repeated graphs, it is easy to verify that the number of $K_{r_{1},r_{2}}$ is half of the above. This completes the proof.
	\end{proof}
	
	\begin{lemma}
		Let $n$, $2s$, $r_1$, $r_2$ and $r$ be positive integers. For positive integer $t$ with $t\leq  s  $, $$h(t)=\begin{pmatrix}2s-t\\r\end{pmatrix}\begin{pmatrix}r\\r_{1}\end{pmatrix} +\sum\limits_{j=1}^{2}\begin{pmatrix}t\\r_{j}\end{pmatrix}\left[\begin{pmatrix}n-r_{j}-1\\r-r_{j}\end{pmatrix}-\begin{pmatrix}2s-t-r_{j}\\r-r_{j}\end{pmatrix}\right]$$ is a convex function of $t$.
	\end{lemma}
	
	\begin{proof}[\rm{\textbf{Proof.}}]
		By direct calculation, we have
		\begin{align*}
			h(t+1)-h(t)&=-\begin{pmatrix}2s-t-1\\r-1\end{pmatrix}\begin{pmatrix}r\\r_{1}\end{pmatrix}+\sum\limits_{j=1}^{2}\begin{pmatrix}t\\r_{j}-1\end{pmatrix}\begin{pmatrix}n-r_{j}-1\\r-r_{j}\end{pmatrix}\\
			&\quad-\sum\limits_{j=1}^{2}\left[\begin{pmatrix}t+1\\r_{j}\end{pmatrix}\begin{pmatrix}2s-t-r_{j}-1\\r-r_{j}\end{pmatrix}-\begin{pmatrix}t\\r_{j}\end{pmatrix}\begin{pmatrix}2s-t-r_{j}\\r-r_{j}\end{pmatrix}\right].
		\end{align*}
		Since $\begin{pmatrix}2s-t-r_{j}\\r-r_{j}\end{pmatrix}=\begin{pmatrix}2s-t-r_{j}-1\\r-r_{j}\end{pmatrix}+\begin{pmatrix}2s-t-r_{j}-1\\r-r_{j}-1\end{pmatrix}$,  we have
		\begin{align*}
			h(t+1)-h(t)&=-\begin{pmatrix}2s-t-1\\r-1\end{pmatrix}\begin{pmatrix}r\\r_{1}\end{pmatrix}+\sum\limits_{j=1}^{2}\begin{pmatrix}t\\r_{j}-1\end{pmatrix}\begin{pmatrix}n-r_{j}-1\\r-r_{j}\end{pmatrix}\\
			&\quad+\sum\limits_{j=1}^{2}\left[\begin{pmatrix}t\\r_{j}\end{pmatrix}\begin{pmatrix}2s-t-r_{j}-1\\r-r_{j}-1\end{pmatrix}-\begin{pmatrix}t\\r_{j}-1\end{pmatrix}\begin{pmatrix}2s-t-r_{j}-1\\r-r_{j}\end{pmatrix}\right].
		\end{align*}
		Similarly, we have
		\begin{align*}
			h(t-1)-h(t)&=\begin{pmatrix}2s-t\\r-1\end{pmatrix}\begin{pmatrix}r\\r_{1}\end{pmatrix}-\sum\limits_{j=1}^{2}\begin{pmatrix}t-1\\r_{j}-1\end{pmatrix}\begin{pmatrix}n-r_{j}-1\\r-r_{j}\end{pmatrix}\\
			&\quad-\sum\limits_{j=1}^{2}\left[
			\begin{pmatrix}t-1\\r_{j}\end{pmatrix}\begin{pmatrix}2s-t-r_{j}\\r-r_{j}-1\end{pmatrix}-
			\begin{pmatrix}t-1\\r_{j}-1\end{pmatrix}\begin{pmatrix}2s-t-r_{j}\\r-r_{j}\end{pmatrix}\right].
		\end{align*}
		Therefore,
		\begin{align*}
			&\quad h(t+1)+h(t-1)-2h(t)\\
			&=\sum\limits_{j=1}^{2}\begin{pmatrix}t-1\\r_{j}-2\end{pmatrix}\begin{pmatrix}n-r_{j}-1\\r-r_{j}\end{pmatrix}+\sum\limits_{j=1}^{2}\left[\begin{pmatrix}t\\r_{j}\end{pmatrix}\begin{pmatrix}2s-t-r_{j}-1\\r-r_{j}-1\end{pmatrix}-\begin{pmatrix}t-1\\r_{j}\end{pmatrix}\begin{pmatrix}2s-t-r_{j}\\r-r_{j}-1\end{pmatrix}\right]\\
			&\quad+\sum\limits_{j=1}^{2}\left[\begin{pmatrix}t-1\\r_{j}-1\end{pmatrix}\begin{pmatrix}2s-t-r_{j}\\r-r_{j}\end{pmatrix}-\begin{pmatrix}t\\r_{j}-1\end{pmatrix}\begin{pmatrix}2s-t-r_{j}-1\\r-r_{j}\end{pmatrix}\right]+\begin{pmatrix}2s-t-1\\r-2\end{pmatrix}\begin{pmatrix}r\\r_{1}\end{pmatrix}.
		\end{align*}
		Using the Pascal's Rule repeatedly, we have
		\begin{align*}
			&~~~~~h(t+1)+h(t-1)-2h(t)\\
			&=\sum\limits_{j=1}^{2}\begin{pmatrix}t-1\\r_{j}-2\end{pmatrix}\begin{pmatrix}n-r_{j}-1\\r-r_{j}\end{pmatrix}+\sum\limits_{j=1}^{2}\left[\begin{pmatrix}t-1\\r_{j}-1\end{pmatrix}\begin{pmatrix}2s-t-r_{j}-1\\r-r_{j}-1\end{pmatrix}-\begin{pmatrix}t-1\\r_{j}\end{pmatrix}\begin{pmatrix}2s-t-r_{j}-1\\r-r_{j}-2\end{pmatrix}\right]\\
			&~~+\sum\limits_{j=1}^{2}\left[\begin{pmatrix}t-1\\r_{j}-1\end{pmatrix}\begin{pmatrix}2s-t-r_{j}-1\\r-r_{j}-1\end{pmatrix}-\begin{pmatrix}t-1\\r_{j}-2\end{pmatrix}\begin{pmatrix}2s-t-r_{j}-1\\r-r_{j}\end{pmatrix}\right]+\begin{pmatrix}2s-t-1\\r-2\end{pmatrix}\begin{pmatrix}r\\r_{1}\end{pmatrix}\\
			&\geq \sum\limits_{j=1}^{2}\begin{pmatrix}t-1\\r_{j}-2\end{pmatrix}\begin{pmatrix}n-r_{j}-1\\r-r_{j}\end{pmatrix}-\sum\limits_{j=1}^{2}\left[\begin{pmatrix}t-1\\r_{j}-2\end{pmatrix}\begin{pmatrix}2s-t-r_{j}-1\\r-r_{j}\end{pmatrix}+\begin{pmatrix}t-1\\r_{j}\end{pmatrix}\begin{pmatrix}2s-t-r_{j}-1\\r-r_{j}-2\end{pmatrix}\right]\\
			&~~+\begin{pmatrix}2s-t-1\\r-2\end{pmatrix}\begin{pmatrix}r\\r_{1}\end{pmatrix}.
		\end{align*}
		Since $n-r_{j}-1\geq 2s-t-r_{j}-1$, we have
		\[
		h(t+1)+h(t-1)-2h(t)\geq \begin{pmatrix}2s-t-1\\r-2\end{pmatrix}\begin{pmatrix}r\\r_{1}\end{pmatrix}-\sum\limits_{j=1}^{2}\begin{pmatrix}t-1\\r_{j}\end{pmatrix}\begin{pmatrix}2s-t-r_{j}-1\\r-r_{j}-2\end{pmatrix}.
		\]
		To prove $h(t+1)+h(t-1)\geq 2h(t)$, we only need to prove
		\begin{align*}	&\quad\quad\quad\qquad\qquad\qquad\qquad\begin{pmatrix}2s-t-1\\r-2\end{pmatrix}\begin{pmatrix}r\\r_{1}\end{pmatrix} \geq \sum\limits_{j=1}^{2}\begin{pmatrix}t-1\\r_{j}\end{pmatrix}\begin{pmatrix}2s-t-r_{j}-1\\r-r_{j}-2\end{pmatrix}\\
			&\Longleftrightarrow\frac{(2s-t-1)!\cdot r!}{(r-2)!\cdot (2s-t-r+1)!\cdot(r-r_1)!\cdot r_{1}!} \geq \sum\limits_{j=1}^{2}\frac{(t-1)!\cdot (2s-t-r_{j}-1)!}{r_{j}!\cdot (t-1-r_{j})!\cdot (r-r_{j}-2)!\cdot (2s-t-r+1)!}\\
			&\Longleftrightarrow \qquad\qquad\qquad\qquad 1\geq \sum\limits_{j=1}^{2}\frac{(t-1)(t-2)\cdots (t-r+r_{j})\cdot r_{j}\cdot(r_{j}-1)}{(2s-t-1)\cdots(2s-t-r+r_{j})\cdot r\cdot (r-1)}.
		\end{align*}
		Since $2s-t\geq t$ and $r=r_{1}+r_{2}$, the last inequality holds. 
		Therefore, $h(t+1)+h(t-1)\geq 2h(t)$, as desired.
	\end{proof}
	By Lemma~2.7, we have the following proposition.
	\begin{proposition}\label{convex2}
		Let $n$, $2s$, $r$, $r_1$ and $r_2$ be positive integers. For positive integer $t$ with $t\leq  s  $, let $g(t)=c\cdot g_{r_1,r_2}(n,s,t)$, where $c=1$ if $r_{1}\neq r_{2}$, and $c=2$ if $r_1=r_2$. Then
		$$g(t)= \begin{pmatrix}2s-t\\r\end{pmatrix}\begin{pmatrix}r\\r_{1}\end{pmatrix} +\sum\limits_{j=1}^{2}\begin{pmatrix}t\\r_{j}\end{pmatrix}\left[\begin{pmatrix}n-r_{j}-1\\r-r_{j}\end{pmatrix}-\begin{pmatrix}2s-t-r_{j}\\r-r_{j}\end{pmatrix}\right]
		+\sum\limits_{j=1}^{2}\begin{pmatrix}\delta\\r_{j}\end{pmatrix}\begin{pmatrix}n-r_j-1\\r-r_{j}-1\end{pmatrix} $$ is a convex function of $t$.
	\end{proposition}

	\begin{proof}[\rm{\textbf{Proof of Theorem~\ref{kr}}}]
		Let $G$ be a graph attaining the maximum number of copies of $K_{r_1,r_2}$ with fractional matching number $s$ and minimum degree $\delta$.
		We first consider the case $r_{1}\neq r_{2}$.
		
		By fractional Tutte-Berge formula, it is not hard to see that $G$ is a subgraph of $K_t\vee (K_{2s-2t}+ \overline{K_{n+t-2s}})$ with $t\le s$ and $2s-2t\neq 1$. Since $\delta(G)=\delta$, it is clear that $\delta\le t$ and hence $G$ is a subgraph of $G(n,s,t)$, $G_{1}$ or $G_{2}$, where $G_{1}\in \mathcal{F}_1(t)$, $G_{2}\in \mathcal{F}_2(t)$. 
		Note that deleting any edge of a graph does not increase the number of copies of $K_{r_1,r_2}$. So we may assume that $G=G(n,s,t)$, $G\in \mathcal{F}_1(t)$ or $G\in \mathcal{F}_2(t)$.
		
		In particular, if $G\in \mathcal{F}_1(s)$, then $\delta=s$, and hence $G=G(n,s,s)$.
		By the maximality of $G$ and Lemma 2.6, we have $G=G(n,s,t)$ for some positive integer $t$ with $\delta\le t\le s$.
		
		{\bf  Case 1. $2s$ is even.}  In this case, $s$ is an integer. By Proposition~\ref{convex2}, we have $t=\delta$ or $t=s$.
		If $t=\delta$, then $G=G(n,s,\delta)$ and hence $N(K_{r_1,r_2},G)=g_{r_1,r_2}(n,s,\delta).$
		If $t=s$, then $G=G(n,s,s)$ and hence $N(K_{r_1,r_2},G)=g_{r_1,r_2}(n,s,s).$
		
		{\bf Case 2. $2s$ is odd.} Since $t\neq s-\frac{1}{2}$, we have $\delta \leq t \leq s-\frac{3}{2}$. By Proposition \ref{convex2}, we have $t=\delta$ or $t=s-\frac{3}{2}$. If $t=\delta$, then $G=G(n,s,\delta)$, and hence $N(K_{r_1,r_2},G)=g_{r_1,r_2}(n,s,\delta)$. If $t=s-\frac{3}{2}$, then $G=G(n,s,s-\frac{3}{2})$ and hence $N(K_{r_1,r_2},G)=g_{r_1,r_2}(n,s,s-\frac{3}{2}).$
		
		For the case $r_{1}= r_{2}$, we have $c=2$ in Eq.~\eqref{2}, by the same discussion, it is easy to verify that the number of copies of $K_{r_{1},r_{2}}$ is as desired. This completes the proof.
	\end{proof}
	
	\section{Conclusion}
	
	In this paper, we have determined the maximum number of copies of $K_{\ell}$ in an $n$-vertex graph with prescribed fractional matching number and minimum degree. 
	Our result yields Shi and Ma's work in \cite{Shi} about the maximum size of graphs with given fractional matching number and minimum degree at least one.
	Moreover, we have used a similar method to determine the maximum number of copies of $K_{r_{1},r_{2}}$ with prescribed fractional matching number and minimum degree.
	
	{\bf Acknowledgement.} The authors are grateful to Professor Xingzhi Zhan for his constant support and guidance. The authors also thank Dr. Leilei Zhang for suggesting the problems investigated in this paper and for helpful comments. This research  was supported by the NSFC grant 12271170 and Science and Technology Commission of Shanghai Municipality (STCSM) grant 22DZ2229014.


\begin{thebibliography}{99}
		\bibitem{alon}N. Alon, P. Frankl, Tur$\acute{\mathrm {a} }$n graphs with bounded matching number, J. Comb. Theory, Ser. B 165 (2024) 223--229.
		
		\bibitem{bondy}
		J.A. Bondy, U.S.R. Murty, Graph Theory, GTM, vol. 244, Springer, Berlin, 2008.
		
		\bibitem{chvatal}V. Chv$\acute{\mathrm{a}}$tal, D. Hanson, Degrees and matchings, J. Comb. Theory, Ser. B 20 (1976) 128--138.
		
		\bibitem{duan}X. Duan, B. Ning, X. Peng, J. Wang, W. Yang, Maximizing the number of cliques in graphs with given matching number, Discrete Appl. Math. 287 (2020) 110--117.
		
		\bibitem{Erdos}P. Erd\H{o}s, T. Gallai., On maximal paths and circuits of graphs, Acta Math. Acad. Sci. Hung. 10 (1959) 337--356.
		
		\bibitem{Gerbner}D. Gerbner, A. Methuku, M. Vizer, Generalized Tur$\acute{\mathrm {a} }$n problems for disjoint copies of graphs, Discrete Math. 342 (2019) 3130--3141.
		
		\bibitem{zll}Y. Liu, L. Zhang, The maximum number of complete multipartite subgraphs in graphs with given circumference or matching number, Discrete Math. 347(1) (2024) 113734.
		
		\bibitem{12}C. Lu, L. Yuan, P. Zhang. The maximum number of copies of $K_{r,s}$ in graphs without long cycles or paths, Electron. J. Combin. 28(4) (2021) 4.4, 15 pp.
		
		\bibitem{Luo}R. Luo, The maximum number of cliques in graphs without long cycles, J. Comb. Theory, Ser. B 128 (2017) 219--226.
		
		\bibitem{Ma}T. Ma, J. Qian, C. Shi, Maximum Size of a graph with given fractional matching number, Electron. J. Combin. 29(3) (2022) 3.55, 13 pp.
		
		\bibitem{Ning}B. Ning, J. Wang, The formula for Tur$\acute{\mathrm {a} }$n number of spanning linear forests, Discrete Math. 343(8) (2020) 111924.
		
		
		\bibitem{Esch}E. Scheinerman, D. Ullman, Fractional Graph Theory: A Rational Approach to the Theory of Graphs, John Wiley, New York, 1997.
		
		\bibitem{Shi}C. Shi, T. Ma, A note on maximum size of a graph without isolated vertices under the given matching number, Appl. Math. Comput. 460 (2024) 128295.
		
		\bibitem{Turan}P. Tur$\acute{\mathrm {a} }$n, On an external problem in graph theory, Mat. Fiz. Lapok, 48 (1941) 436--452.
		
		\bibitem{wang}J. Wang, The shifting method and generalized Tur$\acute{\mathrm {a} }$n number of matchings, Eur. J. Comb. 85 (2020) 7.
		
		\bibitem{zll2}L. Zhang, Further Results on the Generalized Tur$\acute{\mathrm {a} }$n Number of Spanning Linear Forests, Bull. Malays. Math. Sci. Soc. 46(1) (2023) 22.
		
		\bibitem{zlw}L. Zhang, L. Wang, J. Zhou, The generalized Tur$\acute{\mathrm {a} }$n number of spanning linear forests, Graphs Comb. 38 (2022) 40.
	\end{thebibliography}
\end{document}